\documentclass[pdflatex,sn-mathphys-num,alpha]{sn-jnl}

\usepackage{graphicx}%
\usepackage{multirow}%
\usepackage{amsmath,amssymb,amsfonts}%
\usepackage{amsthm}%
\usepackage{mathrsfs}%
\usepackage[title]{appendix}%
\usepackage{xcolor}%
\usepackage{textcomp}%
\usepackage{manyfoot}%
\usepackage{booktabs}%
\usepackage{algorithm}%
\usepackage{algorithmicx}%
\usepackage{algpseudocode}%
\usepackage{listings}%

\usepackage{mathtools}
\mathtoolsset{showonlyrefs} 

\theoremstyle{thmstyleone}%
\newtheorem{theorem}{Theorem}
\newtheorem{lemma}[theorem]{Lemma}%
\newtheorem{corollary}[theorem]{Corollary}%

\theoremstyle{thmstyletwo}%
\newtheorem{example}{Example}%
\newtheorem{remark}{Remark}%

\theoremstyle{thmstylethree}%
\newtheorem{definition}{Definition}%

\newcommand{\bbR}{\mathbb{R}}
\newcommand{\bbC}{\mathbb{C}}
\newcommand{\bbP}{\mathbb{P}}
\newcommand{\rmd}{\mathrm{d}}
\newcommand{\trans}{\mathsf{T}}
\newcommand{\RK}{\mathrm{RK}}
\newcommand{\CSRK}{\mathrm{CSRK}}
\newcommand{\Nplus}{\mathbb{N}_{+}}
\newcommand{\MC}{\mathcal{M}_C}

\newcommand{\Tree}{\protect\scalebox{1.0}{%
\text{\hspace{0.5pt}\rm  T\hspace{0.5pt}}}}

\renewcommand{\Tree}{%
\text{\hspace{0.3pt}\rm  T\hspace{0.3pt}}}

\usepackage{tikz}
\usepackage{pgfplotstable}
\usetikzlibrary{arrows,positioning,plotmarks,external,patterns,angles,
decorations.pathmorphing,backgrounds,fit,shapes,calc}

\newcounter{ncount}
\newcounter{scount}

\def\metrics#1#2#3{%
\def\treesize{#1}\def\thick{#2}\def\rad{#3}}

\def\vertex#1#2{%
\addtocounter{ncount}{1}
\setcounter{scount}{\thencount}
\addtocounter{scount}{-#1}
\node (\thencount) at ($(\thescount) + ( #2,1)$){};
\draw (\thencount)--(\thescount);
\fill (\thencount) circle;
}

\def\treenv#1{%
\begin{tikzpicture}[x=\treesize mm,y=\treesize mm,radius=\rad pt,line width=\thick pt,inner sep=0,baseline=-0.02cm]
\node (0) at (0,0) {}; \fill (0) circle;
\setcounter{ncount}{0}
#1
\end{tikzpicture}}

\metrics{2.5}{1}{2}

\begin{document}

\title[Polynomial moment approach to a CSRK rank condition]{A Polynomial Moment Approach to a Rank Condition for Continuous-Stage Runge--Kutta Methods}

\author*[1]{\fnm{Yuto} \sur{Miyatake}}\email{yuto.miyatake.cmc@osaka-u.ac.jp}

\affil*[1]{\orgdiv{D3 Center}, \orgname{The University of Osaka}, \orgaddress{\country{Japan}}}

\abstract{
In the study of energy-preserving methods for Hamiltonian systems, polynomial continuous-stage Runge--Kutta methods play an important role.
Necessary and sufficient conditions for such methods to be energy-preserving have already been established.
They are energy-preserving if the matrix $M\in \bbR^{s\times s}$ defining the method is symmetric, and the converse holds under the assumption that a certain $s\times \infty$ matrix $\Phi^\CSRK$ has full row rank.  
It was conjectured in~\cite[Remark 3]{MiyatakeButcher2016} that the full-rank assumption should always hold for every consistent polynomial continuous-stage Runge--Kutta method.
In this paper, we prove the conjecture by showing that the matrix $\Phi^\CSRK$ has full row rank under the standard consistency condition.  
The proof is a direct application of the polynomial moment problem solved by Pakovich and Muzychuk~\cite{PakovichMuzychuk2009}.  }



\maketitle

\section{Introduction}\label{sec1}

Continuous-stage Runge--Kutta (CSRK) methods are one-step numerical integrators for the initial value problem of ordinary differential equations
\begin{equation}\label{eq:ivp}
    \dot y=f(y),\qquad y(0)=y_0\in\bbR^N,
\end{equation}
where $f:\bbR^N\to\bbR^N$ is a sufficiently smooth vector field.
In particular, a polynomial CSRK method $y_0\mapsto y_1$ is defined as
\begin{align}
    Y_{\tau}
    &=y_0+h\int_0^1 A_{\tau,\zeta}f(Y_{\zeta})\,\rmd\zeta,
    \qquad 0\le \tau\le1, \label{eq:csrk-stage}\\
    y_1
    &=y_0+h\int_0^1 B_{\tau}f(Y_{\tau})\,\rmd\tau \, (=Y_1), \label{eq:csrk-step}
\end{align}
where $A_{\tau,\zeta}$ is a bivariate polynomial satisfying $A_{0,\zeta}=0$ and $B_{\tau}=A_{1,\tau}$ is a univariate polynomial.  
It is called an $s$-degree polynomial CSRK method if $A_{\tau,\zeta}$ is of degree at most $s$ in $\tau$.
The abscissa polynomial is defined by
\begin{equation}\label{eq:C-def}
    C(\tau)=C_{\tau}:=\int_0^1A_{\tau,\zeta}\,\rmd\zeta.
\end{equation}
Then $A_{0,\zeta}=0$ implies $C(0)=0$.
A method is said to be consistent if
\begin{equation*}
    \int_0^1B_{\tau}\,\rmd\tau=1,
\end{equation*}
or equivalently if
\begin{equation}\label{eq:C-endpoints}
    C(1)=1.
\end{equation}

CSRK methods provide a flexible framework for constructing numerical integrators for ordinary differential equations.
They go back to Butcher's algebraic theory of integration methods~\cite{Butcher1972}; see also Butcher's monograph on B-series~\cite{Butcher2021}.  
In particular, CSRK formulations have appeared in the study of energy-preserving integrators for Hamiltonian systems, including the average-vector-field collocation method and Hamiltonian boundary value methods~\cite{Hairer2010,BrugnanoIavernaroTrigiante2010}, continuous-stage energy-preserving methods~\cite{Miyatake2014,MiyatakeButcher2016}, and related Runge--Kutta type constructions~\cite{TangSun2012,TangSun2014}.

Both Runge--Kutta (RK) methods and CSRK methods are special cases of the more general class of B-series methods, and they share many properties. 
For example, one can apply the B-series theory to both.
In addition, the computational costs of $s$-stage implicit RK methods and $s$-degree polynomial CSRK methods are comparable.
Nevertheless, there are also important differences between the two classes.
From the viewpoint of geometric integration, RK methods can be symplectic for Hamiltonian systems, but they cannot preserve the energy of arbitrary Hamiltonian systems exactly.  
By contrast, CSRK methods can be energy-preserving for Hamiltonian systems, but they cannot preserve the symplectic structure of arbitrary Hamiltonian systems exactly.

It is well known that an $s$-stage RK method is symplectic if 
\begin{equation*}
    b_ia_{ij}+b_ja_{ji}=b_ib_j,
    \qquad i,j=1,\ldots,s,
\end{equation*}
where $a_{ij}$ and $b_i$ are the coefficients of the RK method~\cite{Lasagni1988,SanzSerna1988,Suris1988}.  
This condition is also necessary if the method is irreducible, i.e., if there are no equivalent stages (see~\cite{Hairer1994} and \cite[Section VI.7.3]{HairerLubichWanner2006}).  
The irreducibility condition is equivalent to the full row rank of a certain $s\times \infty$ matrix.

Let 
\begin{equation}
\label{eq:tree}
\Tree =\bigg\{ \raisebox{-6pt}{\treenv{},\,
    \treenv{\vertex10}, \, 
    \treenv{\vertex1{-0.6}\vertex2{0.6}}, \,
    \treenv{\vertex10\vertex10}, \,
    \treenv{\vertex1{-0.7}\vertex20\vertex3{0.7}}, \,
    \treenv{\vertex1{-0.6}\vertex2{0.6}\vertex10}, \,
    \treenv{\vertex10\vertex1{-0.6}\vertex2{0.6}}, \,
    \treenv{\vertex10\vertex10\vertex10},\, \dots} \bigg\}
\end{equation}
be the set of rooted trees.
We define $\varphi_{\RK,i}:\Tree\to\bbR$ as the elementary weight for 
\begin{equation}
\begin{array}{c|c}
    c & A \\
        \hline
        & e_i^\trans A
\end{array}.
\end{equation}
We define an $s\times \infty$ matrix $\Phi^{\RK}$ by
\begin{equation}
    \Phi^{\RK} = 
    \begin{bmatrix}
        \varphi^{\RK}_1(\treenv{}) & \varphi^{\RK}_1(\treenv{\vertex10}) & \varphi^{\RK}_1(\treenv{\vertex1{-0.6}\vertex2{0.6}}) & \cdots \\
        \varphi^{\RK}_2(\treenv{}) & \varphi^{\RK}_2(\treenv{\vertex10}) & \varphi^{\RK}_2(\treenv{\vertex1{-0.6}\vertex2{0.6}}) & \cdots \\
        \vdots & \vdots & \ddots & \ddots \\
        \varphi^{\RK}_s(\treenv{}) & \varphi^{\RK}_s(\treenv{\vertex10}) & \varphi^{\RK}_s(\treenv{\vertex1{-0.6}\vertex2{0.6}}) & \cdots
    \end{bmatrix}.
\end{equation}
Then, an RK method is irreducible if and only if $\Phi^{\RK}$ has full rank $s$ (see~\cite[Section VI.7.3]{HairerLubichWanner2006}).

Interestingly, a similar rank condition appears in the characterization of energy-preserving polynomial CSRK methods~\cite{MiyatakeButcher2016}. 
We represent the kernel as
\begin{equation}\label{eq:A-M}
    A_{\tau,\zeta}
    =
    \begin{bmatrix}
        \tau & \tau^2/2 & \cdots & \tau^s/s
    \end{bmatrix}
    M
    \begin{bmatrix}
        1 & \zeta & \cdots & \zeta^{s-1}
    \end{bmatrix}^\trans,
\end{equation}
where $M\in\bbR^{s\times s}$.  
Then, it is straightforward to verify that the method is energy-preserving if $M$ is symmetric.
Conversely, this condition is also necessary if the method is consistent and the matrix $\Phi^{\CSRK}$ defined below has full row rank $s$:
\begin{equation}\label{eq:PhiCSRK}
    \Phi_{i,k}^{\CSRK}
    =\int_0^1 \tau^{i-1}C_\tau^k\,\rmd\tau,
    \qquad
    i=1,\ldots,s,
    \quad k=1,2,\ldots .
\end{equation}
It was conjectured in~\cite[Remark 3]{MiyatakeButcher2016} that this assumption should always hold for every consistent polynomial CSRK method.

In this paper, we prove in Section~\ref{sec:pmp} that the moment matrix~\eqref{eq:PhiCSRK} has full row rank $s$ under the standard consistency condition~\eqref{eq:C-endpoints}.
We also note that the rank condition is not directly related to the reducibility of continuous-stage methods.
We discuss this distinction in more detail in Section~\ref{sec:pointwise}.

\paragraph{Notation.}
For $m\geq 0$, $\bbP_m$ denotes the real vector space of polynomials of degree at most $m$.
We also write $\Nplus=\{1,2,\ldots\}$.

\section{Polynomial moment problem and proof of the rank condition}\label{sec:pmp}

\subsection{The moment map and the rank condition}

We formulate this rank condition as a moment non-degeneracy statement.  
For a real polynomial $C$, we define
\begin{equation}\label{eq:moment-map}
    \MC:\bbP_{s-1}\to \bbR^{\Nplus},
    \qquad
    q\mapsto
    \left(
        \int_0^1 q(\tau)C(\tau)^k\,\rmd\tau
    \right)_{k\ge1}.
\end{equation}

As Lemma~\ref{lem:moment-rank} below shows,
the full rank condition for the matrix $\Phi^\CSRK$ in~\eqref{eq:PhiCSRK} is equivalent to the injectivity of the map $\MC$. 
Note that if the map~\eqref{eq:moment-map} is not injective, there exists a non-zero polynomial $q\in\bbP_{s-1}$ such that
\begin{equation}\label{eq:moment-null}
    \int_0^1 q(\tau)C(\tau)^k\,\rmd\tau=0,
    \qquad k=1,2,\ldots .
\end{equation}

\begin{lemma}\label{lem:moment-rank}
Let $C$ be a real polynomial, and define $\Phi^\CSRK$ by~\eqref{eq:PhiCSRK}.  The following are equivalent:
\begin{enumerate}
    \item[(i)] the moment map $\MC:\bbP_{s-1}\to\bbR^{\Nplus}$ is injective;
    \item[(ii)] the matrix $\Phi^\CSRK$ has full row rank $s$.
\end{enumerate}
\end{lemma}

\begin{proof}
If $\Phi^\CSRK$ does not have full row rank, there exists $0\ne\alpha=(\alpha_1,\ldots,\alpha_s)^{\trans}\in\bbR^s$ such that $\alpha^{\trans}\Phi^\CSRK=0$.  Set
\begin{equation}
    q(\tau)=\sum_{i=1}^s\alpha_i\tau^{i-1}.
\end{equation}
Then $0\ne q\in\bbP_{s-1}$ and, for all $k\ge1$,
\begin{equation}
    \int_0^1q(\tau)C(\tau)^k\,\rmd\tau
    =\sum_{i=1}^s\alpha_i\int_0^1\tau^{i-1}C(\tau)^k\,\rmd\tau
    =(\alpha^{\trans}\Phi^\CSRK)_k=0.
\end{equation}
Thus $\MC$ is not injective.  Conversely, if $\MC$ is not injective, take $0\ne q\in\bbP_{s-1}$ satisfying~\eqref{eq:moment-null} and write $q(\tau)=\sum_{i=1}^s\alpha_i\tau^{i-1}$.  Then $\alpha\ne0$ and $\alpha^{\trans}\Phi=0$, so $\Phi$ cannot have full row rank.
\end{proof}

\subsection{Full rank property}

The main claim of the paper is the following.

\begin{theorem}\label{thm:auto-nd}
Let $C$ be a real polynomial satisfying
\begin{equation}
    C(0)=0,
    \qquad
    C(1)=1.
\end{equation}
Then, for every $s\ge1$, the moment map $\MC:\bbP_{s-1}\to\bbR^{\Nplus}$ is injective.  Consequently, the matrix $\Phi^\CSRK$ in~\eqref{eq:PhiCSRK} has full row rank $s$ for every consistent polynomial CSRK method. 
\end{theorem}

Theorem~\ref{thm:auto-nd} can be seen as a direct application of the solution of the polynomial moment problem by Pakovich and Muzychuk~\cite{PakovichMuzychuk2009}. 

\begin{theorem}[Pakovich--Muzychuk~\cite{PakovichMuzychuk2009}]\label{thm:PM}
Let $P$ be a complex polynomial, and let $a,b\in\bbC$ with $a\ne b$.
There exists a nonzero complex polynomial $q$ satisfying 
\begin{equation}\label{eq:pmp}
    \int_a^b P(z)^k q(z)\,\rmd z=0,
    \qquad k=0,1,2,\ldots,
\end{equation}
if and only if $P(a)=P(b)$.
\end{theorem}

Theorem~\ref{thm:PM} is a direct consequence of the solution of the polynomial moment problem in~\cite{PakovichMuzychuk2009}.  
More precisely, the main theorem therein expresses every nonzero solution as a sum of reducible solutions associated with compositional right factors $W_j$ satisfying $W_j(a)=W_j(b)$; this forces $P(a)=P(b)$, and the converse is obtained by taking functions of $P$.

\begin{proof}[Proof of Theorem~\ref{thm:auto-nd}]
Suppose that $\MC$ is not injective.  Then there exists $0\ne q\in\bbP_{s-1}$ such that
\begin{equation}\label{eq:q-null-positive}
    \int_0^1 q(\tau)C(\tau)^k\,\rmd\tau=0,
    \qquad k=1,2,\ldots .
\end{equation}
Set
\[
    \widetilde q(\tau)=C(\tau)q(\tau).
\]
Since $C(1)=1$, the polynomial $C$ is not identically zero; hence $\widetilde q$ is non-zero.  Moreover,~\eqref{eq:q-null-positive} gives, for every $m=0,1,2,\ldots$,
\[
    \int_0^1 C(\tau)^m\widetilde q(\tau)\,\rmd\tau
    =\int_0^1q(\tau)C(\tau)^{m+1}\,\rmd\tau
    =0.
\]
Thus $\widetilde q$ is a non-zero solution of the polynomial moment problem~\eqref{eq:pmp} with $P=C$, $a=0$, and $b=1$.  By Theorem~\ref{thm:PM}, this implies $C(0)=C(1)$, contradicting $C(0)=0$ and $C(1)=1$.  Therefore $\MC$ is injective.  The full row rank of $\Phi$ follows from Lemma~\ref{lem:moment-rank}.
\end{proof}

\begin{remark}\label{rem:k-positive}
The moment map~\eqref{eq:moment-map} uses only the moments with $k\ge1$, whereas Theorem~\ref{thm:PM} is stated for $k\ge0$.  The multiplication by $C$ in the proof is the simple device that converts the vanishing of positive moments for $q$ into the vanishing of all moments for $Cq$.
\end{remark}

\subsection{Consequence for energy-preserving methods}\label{sec:energy}

We recall the characterization of energy-preserving polynomial CSRK methods from~\cite{MiyatakeButcher2016}.  Consider Hamiltonian systems
\begin{equation}\label{eq:hamiltonian}
    \dot y=S\nabla H(y),
\end{equation}
where $S$ is a constant non-singular skew-symmetric matrix.  
Combining Theorem 3.5 of~\cite{MiyatakeButcher2016}, summarized in Section~\ref{sec1}, with Theorem~\ref{thm:auto-nd} yields the following characterization of energy-preservation for consistent polynomial CSRK methods.

\begin{corollary}\label{cor:energy}
Let a consistent polynomial CSRK method have the representation~\eqref{eq:A-M}.  Then the method is energy-preserving for all Hamiltonian systems of the form~\eqref{eq:hamiltonian} if and only if
\begin{equation}
    M=M^{\trans}.
\end{equation}
\end{corollary}

\section{Relation with pointwise reducibility}\label{sec:pointwise}

The rank condition studied above is reminiscent of the full-rank characterizations that appear in finite-stage RK theory.  
However, for continuous-stage methods the direct analogue of stage reducibility is pointwise reducibility.  
It is a condition on the equality of stage values, not on the moment matrix~\eqref{eq:PhiCSRK}.

\begin{definition}[Pointwise equivalence and pointwise reducibility]\label{def:pointwise}
Two parameters $\tau_1,\tau_2\in[0,1]$ are called pointwise equivalent if, for every sufficiently smooth initial value problem~\eqref{eq:ivp} and every sufficiently small step size for which the stage equation~\eqref{eq:csrk-stage} is uniquely solvable, the corresponding stage curve satisfies
\begin{equation}
    Y_{\tau_1}=Y_{\tau_2}.
\end{equation}
The method is called pointwise reducible if there exist $\tau_1\ne\tau_2$ that are pointwise equivalent, and pointwise irreducible otherwise.
\end{definition}

Butcher discusses the concepts of equivalence and reducibility for more general integration methods in~\cite[Section 4.4]{Butcher2021}.  
He defines $A$-equivalence, and the pointwise equivalence introduced above is a special case of $A$-equivalence for polynomial CSRK methods.

Pointwise irreducibility is a natural continuous-stage analogue of irreducibility for finite-stage RK methods.  
By contrast, 
the moment non-degeneracy tests whether a non-zero polynomial stage direction can be invisible to all powers of the single abscissa polynomial $C$.  
The pointwise reducibility depends on the full kernel $A_{\tau,\zeta}$, whereas the moment non-degeneracy depends only on $C$.

The following example shows that, although every consistent polynomial CSRK method is moment nondegenerate, such a method may still be pointwise reducible.

\begin{example}[Pointwise reducible while $\Phi^\CSRK$ is full rank]\label{ex:reducible-rho}
Let
\begin{equation}\label{eq:rho-def}
    \rho(\tau)=2\tau^2-\tau
\end{equation}
and define
\begin{equation}\label{eq:A-rho-example}
    A_{\tau,\zeta}
    =\rho(\tau)+\rho(\tau)^2\left(\zeta-\frac12\right).
\end{equation}
Then $A_{0,\zeta}=0$ and
\begin{equation}
    B_{\zeta}=A_{1,\zeta}=\zeta+\frac12,
    \qquad
    \int_0^1B_{\zeta}\,\rmd\zeta=1.
\end{equation}
Thus the method is consistent.  Its abscissa polynomial is
\begin{equation}
    C(\tau)=\int_0^1A_{\tau,\zeta}\,\rmd\zeta=\rho(\tau).
\end{equation} 
For this example, it follows that $\rho(0)=\rho(1/2)=0$, and hence $Y_0=Y_{1/2}=y_0$ for every problem.  
This indicates that the method is pointwise reducible.
\end{example}

Example~\ref{ex:reducible-rho} illustrates why pointwise reducibility and the full-rank property of $\Phi^\CSRK$ should be kept conceptually separate.  
The former identifies redundant stage parameters.
The latter excludes invisible polynomial directions in the moment matrix $\Phi^\CSRK$~\eqref{eq:PhiCSRK}.  
These are different tests, and in the consistent polynomial setting Theorem~\ref{thm:auto-nd} implies that the matrix $\Phi^\CSRK$ is always full rank.

\section{Concluding remarks}\label{sec:conclusion}

We proved that the moment matrix appearing in the characterization of energy-preserving polynomial CSRK methods has full row rank under the standard consistency conditions.  Equivalently, the associated moment map is injective on the polynomial stage-direction space.  
The proof is a direct application of the solution to the polynomial moment problem obtained by Pakovich and Muzychuk.
 Consequently, the full-rank conjecture in~\cite[Remark 3]{MiyatakeButcher2016} follows, and the rank assumption in the necessity part of the energy-preserving condition can be removed.

We also clarified that this moment non-degeneracy is not the same as pointwise irreducibility.  Pointwise reducibility concerns equality of stage values and depends on the full kernel $A_{\tau,\zeta}$.  The rank condition treated in this paper concerns only the abscissa polynomial $C$ and the finite-dimensional polynomial stage-direction space.

\bmhead{Acknowledgements}

The author thanks Norio Nawata for helpful discussions that led to the identification of the connection with the polynomial moment problem and the relevant result in~\cite{PakovichMuzychuk2009}.
The work was supported by JSPS KAKENHI Grant Numbers 21K18301, 24K02951, 24K00540 and 25H00449.

\bibliography{sn-bibliography}


\begin{thebibliography}{14}
\ifx \bisbn   \undefined \def \bisbn  #1{ISBN #1}\fi
\ifx \binits  \undefined \def \binits#1{#1}\fi
\ifx \bauthor  \undefined \def \bauthor#1{#1}\fi
\ifx \batitle  \undefined \def \batitle#1{#1}\fi
\ifx \bjtitle  \undefined \def \bjtitle#1{#1}\fi
\ifx \bvolume  \undefined \def \bvolume#1{\textbf{#1}}\fi
\ifx \byear  \undefined \def \byear#1{#1}\fi
\ifx \bissue  \undefined \def \bissue#1{#1}\fi
\ifx \bfpage  \undefined \def \bfpage#1{#1}\fi
\ifx \blpage  \undefined \def \blpage #1{#1}\fi
\ifx \burl  \undefined \def \burl#1{\textsf{#1}}\fi
\ifx \doiurl  \undefined \def \doiurl#1{\url{https://doi.org/#1}}\fi
\ifx \betal  \undefined \def \betal{\textit{et al.}}\fi
\ifx \binstitute  \undefined \def \binstitute#1{#1}\fi
\ifx \binstitutionaled  \undefined \def \binstitutionaled#1{#1}\fi
\ifx \bctitle  \undefined \def \bctitle#1{#1}\fi
\ifx \beditor  \undefined \def \beditor#1{#1}\fi
\ifx \bpublisher  \undefined \def \bpublisher#1{#1}\fi
\ifx \bbtitle  \undefined \def \bbtitle#1{#1}\fi
\ifx \bedition  \undefined \def \bedition#1{#1}\fi
\ifx \bseriesno  \undefined \def \bseriesno#1{#1}\fi
\ifx \blocation  \undefined \def \blocation#1{#1}\fi
\ifx \bsertitle  \undefined \def \bsertitle#1{#1}\fi
\ifx \bsnm \undefined \def \bsnm#1{#1}\fi
\ifx \bsuffix \undefined \def \bsuffix#1{#1}\fi
\ifx \bparticle \undefined \def \bparticle#1{#1}\fi
\ifx \barticle \undefined \def \barticle#1{#1}\fi
\bibcommenthead
\ifx \bconfdate \undefined \def \bconfdate #1{#1}\fi
\ifx \botherref \undefined \def \botherref #1{#1}\fi
\ifx \url \undefined \def \url#1{\textsf{#1}}\fi
\ifx \bchapter \undefined \def \bchapter#1{#1}\fi
\ifx \bbook \undefined \def \bbook#1{#1}\fi
\ifx \bcomment \undefined \def \bcomment#1{#1}\fi
\ifx \oauthor \undefined \def \oauthor#1{#1}\fi
\ifx \citeauthoryear \undefined \def \citeauthoryear#1{#1}\fi
\ifx \endbibitem  \undefined \def \endbibitem {}\fi
\ifx \bconflocation  \undefined \def \bconflocation#1{#1}\fi
\ifx \arxivurl  \undefined \def \arxivurl#1{\textsf{#1}}\fi
\csname PreBibitemsHook\endcsname

\bibitem[\protect\citeauthoryear{Miyatake and Butcher}{2016}]{MiyatakeButcher2016}
\begin{barticle}
\bauthor{\bsnm{Miyatake}, \binits{Y.}},
\bauthor{\bsnm{Butcher}, \binits{J.C.}}:
\batitle{A characterization of energy-preserving methods and the construction of parallel integrators for {H}amiltonian systems}.
\bjtitle{SIAM J. Numer. Anal.}
\bvolume{54}(\bissue{3}),
\bfpage{1993}--\blpage{2013}
(\byear{2016})
\doiurl{10.1137/15M1020861}
\end{barticle}
\endbibitem

\bibitem[\protect\citeauthoryear{Pakovich and Muzychuk}{2009}]{PakovichMuzychuk2009}
\begin{barticle}
\bauthor{\bsnm{Pakovich}, \binits{F.}},
\bauthor{\bsnm{Muzychuk}, \binits{M.}}:
\batitle{Solution of the polynomial moment problem}.
\bjtitle{Proc. Lond. Math. Soc. (3)}
\bvolume{99}(\bissue{3}),
\bfpage{633}--\blpage{657}
(\byear{2009})
\doiurl{10.1112/plms/pdp010}
\end{barticle}
\endbibitem

\bibitem[\protect\citeauthoryear{Butcher}{1972}]{Butcher1972}
\begin{barticle}
\bauthor{\bsnm{Butcher}, \binits{J.C.}}:
\batitle{An algebraic theory of integration methods}.
\bjtitle{Math. Comp.}
\bvolume{26},
\bfpage{79}--\blpage{106}
(\byear{1972})
\doiurl{10.2307/2004720}
\end{barticle}
\endbibitem

\bibitem[\protect\citeauthoryear{Butcher}{2021}]{Butcher2021}
\begin{bbook}
\bauthor{\bsnm{Butcher}, \binits{J.C.}}:
\bbtitle{B-series---{A}lgebraic Analysis of Numerical Methods}.
\bsertitle{Springer Series in Computational Mathematics},
vol. \bseriesno{55},
p. \bfpage{310}.
\bpublisher{Springer}, 
(\byear{2021}).
\doiurl{10.1007/978-3-030-70956-3} .
\burl{https://doi.org/10.1007/978-3-030-70956-3}
\end{bbook}
\endbibitem

\bibitem[\protect\citeauthoryear{Hairer}{2010}]{Hairer2010}
\begin{barticle}
\bauthor{\bsnm{Hairer}, \binits{E.}}:
\batitle{Energy-preserving variant of collocation methods}.
\bjtitle{JNAIAM. J. Numer. Anal. Ind. Appl. Math.}
\bvolume{5}(\bissue{1-2}),
\bfpage{73}--\blpage{84}
(\byear{2010})
\end{barticle}
\endbibitem

\bibitem[\protect\citeauthoryear{Brugnano et~al.}{2010}]{BrugnanoIavernaroTrigiante2010}
\begin{barticle}
\bauthor{\bsnm{Brugnano}, \binits{L.}},
\bauthor{\bsnm{Iavernaro}, \binits{F.}},
\bauthor{\bsnm{Trigiante}, \binits{D.}}:
\batitle{Hamiltonian boundary value methods (energy preserving discrete line integral methods)}.
\bjtitle{JNAIAM. J. Numer. Anal. Ind. Appl. Math.}
\bvolume{5}(\bissue{1-2}),
\bfpage{17}--\blpage{37}
(\byear{2010})
\end{barticle}
\endbibitem

\bibitem[\protect\citeauthoryear{Miyatake}{2014}]{Miyatake2014}
\begin{barticle}
\bauthor{\bsnm{Miyatake}, \binits{Y.}}:
\batitle{An energy-preserving exponentially-fitted continuous stage {R}unge--{K}utta method for {H}amiltonian systems}.
\bjtitle{BIT}
\bvolume{54}(\bissue{3}),
\bfpage{777}--\blpage{799}
(\byear{2014})
\doiurl{10.1007/s10543-014-0474-4}
\end{barticle}
\endbibitem

\bibitem[\protect\citeauthoryear{Tang and Sun}{2012}]{TangSun2012}
\begin{barticle}
\bauthor{\bsnm{Tang}, \binits{W.}},
\bauthor{\bsnm{Sun}, \binits{Y.}}:
\batitle{Time finite element methods: a unified framework for numerical discretizations of {ODE}s}.
\bjtitle{Appl. Math. Comput.}
\bvolume{219}(\bissue{4}),
\bfpage{2158}--\blpage{2179}
(\byear{2012})
\doiurl{10.1016/j.amc.2012.08.062}
\end{barticle}
\endbibitem

\bibitem[\protect\citeauthoryear{Tang and Sun}{2014}]{TangSun2014}
\begin{barticle}
\bauthor{\bsnm{Tang}, \binits{W.}},
\bauthor{\bsnm{Sun}, \binits{Y.}}:
\batitle{Construction of {R}unge--{K}utta type methods for solving ordinary differential equations}.
\bjtitle{Appl. Math. Comput.}
\bvolume{234},
\bfpage{179}--\blpage{191}
(\byear{2014})
\doiurl{10.1016/j.amc.2014.02.042}
\end{barticle}
\endbibitem

\bibitem[\protect\citeauthoryear{Lasagni}{1988}]{Lasagni1988}
\begin{barticle}
\bauthor{\bsnm{Lasagni}, \binits{F.M.}}:
\batitle{Canonical {R}unge-{K}utta methods}.
\bjtitle{Z. Angew. Math. Phys.}
\bvolume{39}(\bissue{6}),
\bfpage{952}--\blpage{953}
(\byear{1988})
\doiurl{10.1007/BF00945133}
\end{barticle}
\endbibitem

\bibitem[\protect\citeauthoryear{Sanz-Serna}{1988}]{SanzSerna1988}
\begin{barticle}
\bauthor{\bsnm{Sanz-Serna}, \binits{J.M.}}:
\batitle{Runge-{K}utta schemes for {H}amiltonian systems}.
\bjtitle{BIT}
\bvolume{28}(\bissue{4}),
\bfpage{877}--\blpage{883}
(\byear{1988})
\doiurl{10.1007/BF01954907}
\end{barticle}
\endbibitem

\bibitem[\protect\citeauthoryear{Suris}{1988}]{Suris1988}
\begin{bchapter}
\bauthor{\bsnm{Suris}, \binits{Y.B.}}:
\bctitle{On the conservation of the symplectic structure in the numerical solution of {H}amiltonian systems (in russian)}.
In: \bbtitle{Numerical Solution of Ordinary Differential Equations},
pp. \bfpage{148}--\blpage{160}.
\bpublisher{Keldysh Institute of Applied Mathematics, USSR Academy of Sciences, Moscow}, 
(\byear{1988})
\end{bchapter}
\endbibitem

\bibitem[\protect\citeauthoryear{Hairer}{1994}]{Hairer1994}
\begin{botherref}
\oauthor{\bsnm{Hairer}, \binits{E.}}:
Backward analysis of numerical integrators and symplectic methods.
vol. 1,
pp. 107--132
(1994)
\end{botherref}
\endbibitem

\bibitem[\protect\citeauthoryear{Hairer et~al.}{2006}]{HairerLubichWanner2006}
\begin{bbook}
\bauthor{\bsnm{Hairer}, \binits{E.}},
\bauthor{\bsnm{Lubich}, \binits{C.}},
\bauthor{\bsnm{Wanner}, \binits{G.}}:
\bbtitle{Geometric Numerical Integration},
\bedition{2}nd edn.
\bsertitle{Springer Series in Computational Mathematics},
vol. \bseriesno{31},
p. \bfpage{644}.
\bpublisher{Springer}, 
(\byear{2006})
\end{bbook}
\endbibitem

\end{thebibliography}

\end{document}